\documentclass[12pt]{article}

\usepackage{amsmath,amsthm,amssymb,graphicx}

\def\N{\mathbb N}
\def\N0{\mathbb N_0}

\def\real{\mathbb R}
\def\complex{\mathbb C}

\def\graph{{\mathcal G}}
\def\vertexset{{\mathcal V}}
\def\edgeset{{\mathcal E}}
\def\tree{{\mathcal T}}

\def\edgelength{l}

\def\vertspace{\mathbb H}
\def\espace{\mathbb S}

\def\dop{{\mathcal L}}
\def\domain{{\mathcal D}}

\def\ft{{\mathcal F}}

\title{\bf The Continuous Graph FFT}
\author{Robert Carlson \\
Department of Mathematics \\ 
University of Colorado at Colorado Springs \\
rcarlson@uccs.edu}

\newtheorem{thm}{Theorem}[section]

\newtheorem{lem}[thm]{Lemma}
\newtheorem{prop}[thm]{Proposition}

\theoremstyle{definition}

\theoremstyle{remark}

\newcommand{\thmref}[1]{Theorem~\ref{#1}}

\newcommand{\lemref}[1]{Lemma~\ref{#1}}
\newcommand{\propref}[1]{Proposition~\ref{#1}}

 \numberwithin{equation}{section}

\begin{document}

\maketitle

\date{}

\begin{abstract}
The discrete Fourier transform and the FFT algorithm are extended from 
the circle to continuous graphs with equal edge lengths.
\end{abstract}

\vskip 25pt

2000 Mathematics Subject Classification 65T50, 34B45

\vskip 25pt

\vspace{1cm} {\underline{Acknowledgements.}\newline

This work was partly supported by Grant UKM2-2811-OD-06 of the U.S. Civilian
Research and Development Foundation.

\newpage

\section{Introduction}

The discrete Fourier transform (DFT) and algorithms for its efficient computation (FFT)
enjoy an enormous range of applications.  One might roughly divide such applications into the
analysis of data which is sampled in time or in space.  
Applications involving spatially sampled data are basic 
in numerical analysis, since constant coefficient
partial differential equations and some of their discrete analogs,
including the heat, wave, Schrodinger, and beam equations in one space dimension,
have exact solutions in Fourier series.  
Other applications of the DFT involving the
analysis of spatially sampled data include noise removal, data compression,     
data interpolation, and approximation of functions.
These applications often take advantage of
the tight and explicit linkage binding the harmonic analysis of uniformly sampled data  
with the harmonic analysis of periodic functions defined on $\real $.

Many aspects of Fourier analysis can be developed through the spectral theory of the second
derivative operator $-D^2 = -d^2/dx^2$ acting on the Hilbert space of square integrable functions
with period $1$.  An orthogonal basis of eigenfunctions is given by $\exp (2\pi ikx)$.    
If $N$ uniformly spaced samples $x_n = n/N$ for $n = 0,\dots ,N-1$ are considered,
the restriction of the exponential functions to the sample points 
\[ \exp (2\pi ikn/N) , \quad k,n = 0,\dots ,N-1, \]
again provides an orthogonal basis for functions on the set $\{ x_n \}$.
These sampled exponential functions are also eigenfunctions for 'local' operators
such as $\Delta _N $, which acts by
\[ \Delta _N f(x_n) = f(x_n) - \frac{1}{2}[f(x_{n-1} + f(x_{n+1})],\]
index arithmetic being carried out modulo $N$.
The fact that efficient FFT algorithms are available for sampled data is not a consequence 
of abstract spectral theory, but
relies on the exponential form of the eigenfunctions and the arithmetic progression
of the frequencies $k$.  

Our discussion now shifts to graphs and their refinements, with the circle and its
uniform sampling serving as an example.  Discrete graphs with $N_{\vertexset} < \infty $ 
vertices and $N_{\edgeset}$ edges have a well developed spectral theory \cite{Chung} based on the 
adjacency or Laplace operators acting on functions defined on the vertex set $\vertexset$.
While these ideas play a role here, the lead actor is a continuous graph, also known as
a metric graph or, when a differential operator is emphasized, a quantum graph.  
Here the edges of a graph are identified with intervals $[a_n,b_n] \subset \real$,
functions are elements of the Hilbert space $\oplus _n L^2[a_n,b_n]$, and differential operators
such as $-D^2$ provide a basis for both physical modeling and harmonic analysis.   
(The terminology for such 'continuous graphs' is not settled, with 
the terms topological graphs or networks also used.)   

The natural and technological worlds offer numerous opportunities for graph modeling,
where applications such as those mentioned above can be considered with 
data sampled from continuous graphs.  These include the sampling of populations along river systems,
and the description of traffic density on road networks.  Biological systems
transport nutrients, waste, heat, and pressure waves through vascular networks, and
electrochemical signals through natural neural networks.  Elaborate networks are   
manufactured in microelectronics, and for microfluidic laboratories on chips.

Despite a rapidly growing literature on 'quantum graphs' \cite{Collect3,Collect1,Collect2}, 
and some work \cite{Baker} generalizing the classical 
theory of Fourier series to functions defined on a continuous graph, 
there is a very limited literature considering the harmonic analysis of sampled data
for these geometric objects.  The papers \cite{Pesenson1, Pesenson2} consider  
sampling on continuous graphs with virtually no length restrictions on the edges.
In contrast, this work exploits the features appearing when graphs have equal length
edges, and obtains the following conclusion:  
every finite continuous graph with edges of equal length
admits a family of DFTs closely analogous to that of the circle, and an FFT algorithm
for their efficient computation.

To unify the presentation a bit, we focus on simple discrete graphs, 
without loops or multiple edges between vertices, whose vertices have degree at least two.  
Given a continuous version of a nonsimple graph without boundary vertices, 
one can insert additional vertices of degree $2$ to reduce to the simple case.  
This insertion of 'invisible' vertices can also be used to reduce graphs whose edge lengths are integer multiples
of a common value to the equal length case.  Algorithmically, this
reduction increases the size of the discrete graph spectral problem whose solution 
is an important component of the DFT.  

There are four subsequent sections in the paper.  The second section starts with a review
of differential operators on continuous graphs.
This leads to a more detailed discussion of the spectral theory of the standard Laplace
differential operator on continuous graphs with equal length edges.  
Most of the material in this section was previously known \cite{Below, Friedman}, 
although Theorem 2.9 appears to be new.
The third section explores the linked spectral theory of Laplacians for continuous graphs 
and their uniformly sampled subgraphs.  The fourth section shows that efficient algorithms
for Fourier analysis are available.  The final section presents a simple example 
where many of the computations can be done 'by hand'.

\section{Spectra of continuous graph Laplacians}

\subsection{Continuous graphs and the standard Laplacian}

In this work a discrete graph $\graph $ will be finite and simple,
with a vertexset $\vertexset$ having $N_{\vertexset}$ points, and 
a set $\edgeset$ of $N_{\edgeset}$ edges.  
Each edge $e$ has a positive length $\edgelength _e$.
It is convenient to number the edges.  For $n = 1,\dots ,N_{\edgeset}$
the edge $e_n$ is identified with a real interval $[a_n,b_n]$ of length $\edgelength _n$.
The resulting topological graph is also denoted $\graph $.
In an obvious fashion one may extend the standard metric and Lebesgue measure 
from the edges to $\graph $.

The identification of graph edges and intervals allows us to define
the Hilbert space 
\[ L^2(\graph ) = \oplus_{n=1}^{N_{\edgeset}} L^2[a_n,b_n] \]  
with the inner product
\[ \langle f, g \rangle  = \int_{\graph } f\overline g
= \sum_n \int_{a_n}^{b_n} f_n(x)\overline{g_n(x)} \ dx , 
\quad f_n:[a_n,b_n] \to \complex . \]
We will employ the standard vertex conditions, requiring continuity at the vertices,
\begin{equation} \label{continuity}
\lim_{x \in e(i) \to v}f(x) = \lim_{x \in e(j) \to v}f(x), 
\quad e(i),e(j) \sim v,
\end{equation}
and the derivative condition
\begin{equation} \label{derive}
\sum_{e_n \sim v} \partial _{\nu} f_n(v) = 0.
\end{equation}
Here the derivative $\partial _{\nu} f_n(v) $ 
is $f_n'(a_n)$, respectively $-f_n'(b_n)$ if $a_n$, respectively $b_n$,
is identified with $v$.

The standard vertex conditions are used to define
the Laplacian $\dop f = -f''$ 
of the continuous graph $\graph $. 
Let $\domain _{max}$ denote the set of functions $f\in L^2(\graph )$ 
with $f'$ absolutely continuous on each $e_n$, and $f'' \in  L^2(\graph )$.  
The domain $\domain $ of $\dop $ is then the set of 
$f \in \domain _{max}$ satisfying the standard conditions \eqref{continuity}
and \eqref{derive}.  From the classical theory of ordinary differential operators 
\cite[p. S123]{Kuchment04} one knows that $\dop $ is self adjoint
with compact resolvent.  This operator is nonnegative, with $0$ being a simple eigenvalue
if $\graph $ is connected.
Writing eigenvalues with multiplicity, the spectrum is thus a sequence 
$0 = \lambda _0 < \lambda _1 \le \lambda _2 \le \dots $, and there is a orthonormal
basis of eigenfunctions. 

We now impose the additional requirement that all edges of $\graph $ have equal length.
In this situation the operator $\dop $ will be denoted $\Delta _{\infty}$.
Initially the edge lengths are taken to be $1$, but the rescaling $\xi = x/L$ converts the system
$D_x^2 Y = \lambda Y$ to
\begin{equation} \label{rescale}
D_{\xi}^2 Y = \mu Y, \quad 0 \le \xi \le 1 , \quad \mu =  L^2\lambda .
\end{equation}
Values of eigenfunctions at graph vertices are not effected, and derivatives
are scaled by $L$.  Thus eigenfunctions with eigenvalue $\lambda $ 
for the standard Laplacian 
on the graph with edge lengths $L$
are taken to eigenfunctions with eigenvalue $L^2 \lambda $ 
for the standard Laplacian on the graph with edge lengths $1$.

The spectrum of $\Delta _{\infty}$ is tightly linked to the discrete graph $\graph $
and the spectrum of the discrete Laplacian of combinatorial graph theory.  
Let $E(\lambda )$ denote the eigenspace of $\Delta _{\infty }$ for the eigenvalue
$\lambda $.  The connections between $E(\lambda )$ and the discrete graph $\graph $
arise in two ways, depending on whether
or not $\lambda \in \{ n^2 \pi ^2 \ | \ n = 1,2,3,\dots \} $.  
Before exploring this dichotomy, we note the following result.

\begin{prop} \label{specper}
If $\omega > 0$, 
\[ \dim E(\omega ^2) = \dim E([\omega + 2n\pi ]^2), \quad n  = 1,2,3,\dots . \] 
\end{prop}

\begin{proof}
Suppose $n \ge 1$ and $Y$ is an eigenfunction with eigenvalue $\omega ^2$. 
On each edge
\[y(x) = A\cos(\omega x) + B \sin(\omega x).\]
The function 
\[y_1(x) = A\cos([\omega + 2n \pi] x) + B \sin([\omega + 2n \pi ] x)\]
then satisfies the equation,
\[-y_1'' = [\omega + 2n\pi ]^2y_1,\]
and at the endpoints of the edge we have
\[y(0) = y_1(0), \quad y(1) = y_1(1). \]
The endpoint derivatives are 
\[y'(0) = \omega  B, \quad y'(1) = \omega [B \cos (\omega ) - A \sin(\omega )] \]
and
\[y_1'(0) = [\omega + 2n\pi ] B, \quad 
y_1'(1) = [\omega + 2n\pi ] [B \cos (\omega ) - A \sin(\omega )] , \]
so
\[ y_1'(0) = \frac{\omega + 2n\pi }{\omega }y'(0) , 
\quad y_1'(1) = \frac{\omega + 2n\pi }{\omega }y'(1).\]
 
Thus $y_1$ satisfies the same interior vertex conditions that $y$ does.  
Moreover, the linear map taking $y \to y_1$ is one-to-one, as is
the analogous map from $E([\omega + 2n\pi ]^2)$ to $E(\omega ^2)$. 

\end{proof}

\subsection{Role of the discrete Laplacian}

Given a vertex $v \in \graph $, let $u_1,\dots ,u_{deg (v)}$ be the vertices adjacent to $v$.
A discrete graph carries a number of linear operators acting on 
the vertex space $\vertspace $ of functions $f:\vertexset \to \complex $,
including the adjacency operator  
\[Af(v) = \sum_{i=1}^{deg(v)} f(u_i),\]
and the degree operator  
\[Tf(v) = deg(v) f(v).\]
Define the operator $\Delta _1$ by 
\[\Delta _1 f(v) = f(v) - T^{-1} Af(v) .\]
$\Delta _1$ is similar to the much studied Laplacian \cite[p. 3]{Chung}, 
\[I - T^{-1/2} AT^{-1/2} .\]

The distinction of the cases $\lambda \in \{ n^2 \pi ^2 \} $ is related to the following
fact.

\begin{lem} \label{endmap}
Fix $\lambda \in \complex$, and consider the vector space of solutions of $-y'' = \lambda y$ 
on the interval $[0,1]$.  The linear function taking $y(x)$ to $(y(0),y(1))$ is an isomorphism  
if and only if $\lambda \notin \{ n^2 \pi ^2 \ | \ n = 1,2,3,\dots \} $.
\end{lem}  

\begin{proof}
For $\lambda \notin \{ n^2 \pi ^2 \} $, the formula
\begin{equation} \label{eform}
y(x,\lambda ) = y(0)\cos(\omega x) 
+ [y(1) - y(0)\cos(\omega )]\frac{\sin (\omega x)}{\sin(\omega )}
\end{equation}
shows that the map is surjective.  
On the other hand, if $\lambda \in \{ n^2 \pi ^2 \} $, then $y(0) = 0$
implies $y(1) = 0$, since $y(x) = B\sin (n\pi x)$.

\end{proof}

As an immediate consequence we have the following result for graphs.

\begin{lem} \label{vanishing}
Suppose the edges of $\graph $ have length $1$, and $\lambda \notin \{ n^2 \pi ^2 \ | \ n = 1,2,3,\dots \} $.
Let $y: \graph \to \complex $ be continuous, and satisfy $-y'' = \lambda y$ on the edges.
If $y(v) = 0$ at all vertices of $\graph $, then $y(x) = 0$ for all $x \in \graph $. 
\end{lem}  

\begin{thm} \label{ceqn}
Suppose $\lambda \notin \{ n^2 \pi ^2 \ | \ n = 1,2,3,\dots \} $
and $y$ is an eigenfunction for $\Delta _{\infty}$.
If $v$ has adjacent vertices $u_1,\dots ,u_{\deg (v)}$, then
\begin{equation} \label{combeqn}
\cos(\omega )y(v) = \frac{1}{\deg (v)} \sum_{i=1}^{\deg (v)} y(u_i). 
\end{equation}
\end{thm}

\begin{proof}
In local coordinates identifying each $u_i$ with $0$,  
\eqref{eform} for $y_i$ on the edges $e_i = (u_i,v)$ gives
\begin{equation} \label{vertexform}
 y_i'(v) = -\omega \sin(\omega )y_i(u_i) 
+ [y_i(v) - y_i(u_i)\cos(\omega )]\frac{\omega \cos(\omega )}{\sin(\omega )}.
\end{equation}
Summing over $i$, the derivative condition at $v$ then gives
\[ 0 = \sum_i y_i'(v) = 
-\omega \sin(\omega )\sum_i y_i(u_i) 
+ \frac{\omega \cos(\omega )}{\sin(\omega )}\sum_i [y_i(v) - y_i(u_i)\cos(\omega )] .\]
Using the continuity of $y$ at $v$ and elementary manipulations gives \eqref{combeqn}.
\end{proof}

\noindent \eqref{combeqn} is clearly an eigenvalue equation, with eigenvalue $\cos(\omega )$, for the linear 
operator $T^{-1}A$ acting on the space of (real or complex valued) functions on the vertex set.  
With this background established, we are ready to relate the spectra of $\Delta _1$ and $\Delta _{\infty}$.

\begin{thm} \label{combspec}
If $\lambda \notin \{ n^2 \pi ^2 \ | \ n = 0,1,2,\dots \} $,
then $\lambda $ is an eigenvalue of $\Delta _{\infty}$
if and only if $1 - \cos (\omega ) = 1 - \cos (\sqrt{\lambda } )$
is an eigenvalue of $\Delta _1$, with the same geometric multiplicity.
\end{thm}  

\begin{proof}
Since $\Delta _1 = I - T^{-1}A$, we may work with $T^{-1}A$.
Suppose first that $y(x,\lambda )$ is an eigenfunction of $\Delta _{\infty}$ satisfying the given vertex conditions. 
\thmref{ceqn} shows that the (linear) evaluation map taking $y:\graph \to \complex $ to 
$y:\vertexset \to \complex $ 
takes eigenfunctions to solutions of \eqref{combeqn}.
By \lemref{vanishing} the kernel of this map is the zero function, so the map is injective.

Suppose conversely that $y:\vertexset \to \complex $ satisfies 
\[ T^{-1}A y (v) = \mu y(v), \quad | \mu | < 1. \]
Pick $\lambda \in \cos^{-1}(\mu )$.  By \lemref{endmap} the function $y:\vertexset \to \complex $ 
extends to a unique continuous function $y(x,\lambda ): \graph \to \complex $ 
satisfying $-y'' = \lambda y$ on each edge.

In local coordinates identifying $v$ with $0$ for each edge $e_i = (v,u_i)$ incident on $v$,
this extended function satisfies \eqref{vertexform}.
Summing gives
\[\sum_i y_i'(v) = \frac{\omega}{\sin(\omega )} [-\sin^2 (\omega )-\cos^2 (\omega )] \sum_i y_i(u_i) 
+ \sum_i y_i(v) \frac{\omega \cos(\omega )}{\sin(\omega )} \]
\[ = -\frac{\omega}{\sin(\omega )} \sum_i y_i(u_i) 
+ \deg (v) y(v) \frac{\omega \cos(\omega )}{\sin(\omega )}. \]
The vertex values satisfy \eqref{combeqn}, so
\[\sum_i y_i'(v) = -\frac{\omega}{\sin(\omega )} \deg (v) \cos(\omega )y(v) 
+ \deg (v) y(v) \frac{\omega \cos(\omega )}{\sin(\omega )} = 0. \]
Thus the extended functions are eigenfunctions of $\Delta _{\infty}$ satisfying the standard vertex conditions.
Since the extension map is linear, and the kernel is the zero function, this map    
is also injective.  

\end{proof}

\subsection{Eigenfunctions at $n^2 \pi ^2$}

Now we turn to eigenvalues $\lambda \in \{ n^2 \pi ^2 \} $.
First recall \cite[p. 7]{Chung} that for both $\Delta _1$ and $\Delta _{\infty}$
$0$ is an eigenvalue whose eigenspace is spanned by functions which 
are constant on connected components of $\graph $.  

For $n \ge 1$ these eigenspaces for $\Delta _{\infty}$ also have a combinatorial interpretation,
closely related to the cycles in $\graph $.  If $C$ is a cycle, and $x$ is distance
along the cycle starting at some selected vertex, 
then the function $\sin (2n\pi x)$ is an eigenfunction of $\Delta _{\infty}$.  
Similarly, if $C$ is an even cycle, then $\sin(n\pi x)$ gives a similar eigenfunction.  
We can make the following observation.

\begin{lem} \label{zerolem}
Suppose $\graph $ is connected, and $\psi $ is an eigenfunction of $\Delta _{\infty}$
with eigenvalue $\lambda = n^2\pi^2$ for $n \ge 1$. 
If $\psi $ vanishes at any vertex, then $\psi $ vanishes at all vertices.  
\end{lem} 

\begin{proof}
Suppose $\psi (v) = 0$ for some vertex $v$.  On any   
edge incident on $v$, the eigenfunction is a linear combination 
\[ \psi (x) = A\cos(n\pi x) + B\sin(n\pi x), \quad v \simeq 0,\]
and clearly $A = 0$.  Since all edge lengths are $1$, at all adjacent vertices $w$, we then have
$\psi (w) = B\sin(n\pi ) = 0$.  By continuity of $\psi $ and 
connectivity of the graph, $\psi $ vanishes at all vertices.
\end{proof}

The next result explores the existence of eigenfunctions vanishing at no vertices.

\begin{lem} \label{nzcnt}
If $\lambda = (2n\pi )^2 $, for $n = 1,2,3,\dots $,
then $\Delta _{\infty}$ has an eigenfunction vanishing at no vertices.

If $\lambda = (2n-1)^2 \pi ^2$, $n = 1,2,3,\dots $,
then $\Delta _{\infty}$ has an eigenfunction vanishing at no
vertices if and only if $\graph$ is bipartite.
\end{lem}

\begin{proof}  
If $\lambda = (2n \pi )^2$, then the desired eigenfunction is simply
$\cos(2n\pi x)$ in local coordinates on each edge.

Suppose $\graph$ is bipartite, with the two classes of vertices labelled $0$ and $1$.  
Pick local coordinates on each edge consistent with the vertex class labels, and
define the eigenfunction to be $\cos([2n-1]\pi x)$.

Suppose conversely that for some $\lambda = (2n-1)^2 \pi ^2$, 
there is an eigenfunction $\psi $ vanishing at no
vertex.  Label the vertices $v$ according to the sign of $\psi (v)$.
In local coordinates for an edge, 
\[\psi (x) = a\cos ([2n-1]\pi x) + b  \sin ([2n-1]\pi x) , \quad a \not= 0.\]
Then if $w$ is a vertex adjacent to $v$ we see that $\psi (w) = -\psi (v)$,
showing that vertices are only adjacent to vertices of opposite sign,
and $\graph $ is bipartite.

\end{proof}

In addition to the vertex space mentioned above, an edge space may be constructed
using the edges of a graph as a basis (we assume the field is $\real $
or $\complex$).  The edge space has the cycle subspace $Z_0(\graph )$ generated 
by cycles, with dimension $N_E - N_V + 1$ \cite[pp. 51--58]{Bollobas} or \cite[pp. 23--28]{Diestel}.  
Let $E_0(n^2\pi ^2) \subset E(n^2\pi ^2)$ be those eigenfunctions of $\Delta _{\infty}$ 
vanishing at the vertices.

\begin{thm} \label{zcnt}
\[ {\rm dim} Z_0(\graph ) = {\rm dim} E_0(4n^2\pi ^2). \]
\end{thm}

\begin{proof}
It suffices to prove the result for a connected graph.
Suppose $Z_0(\graph )$ has dimension $M$.  Picking a spanning tree $\tree $ for $\graph $,
there is \cite[p. 53]{Bollobas} a basis of cycles $C_1,\dots ,C_{M} $ such that each $C_j$ contains
an edge $e_j \notin \tree $, with $e_j$ not contained in any other $C_i$.
Fix $n \in \{ 1,2,3,\dots \}$ and construct eigenfunctions $f_j = \sin(2\pi nx)$ on the edges of $C_j$, 
and $0$ on all other edges.  Here $x$ denotes distance along the cycle starting at some selected vertex. 
If a linear combination $\sum a_if_i$ is $0$, then for $x \in e_j$
\[ 0 = \sum_{i=1}^M a_if_i(x) = a_jf_j(x) , \quad j = 1,\dots ,M. \]
There are $x \in e_j$ where $f_j(x) \not= 0$, so $a_j = 0$ and the functions $f_i$ are independent.
This shows ${\rm dim}E_0(4n^2\pi ^2) \ge {\rm dim}Z_0(\graph )$.

Now suppose that $\psi \in E_0(4n^2\pi ^2)$. After subtracting a linear combination 
$ \sum a_if_i $ we may assume that $\psi $ vanishes on all 
edges not in the spanning tree $\tree $.
For every boundary vertex $v$ of $\tree$,
$\psi $ vanishes identically on all but one edge of $\graph $ incident on $v$,
and by the vertex conditions it then vanishes on all edges incident on $v$.  
Continuing away from the boundary of the spanning tree,
we see that $\psi$ is the $0$ function.  

\end{proof}

\vskip 5pt
\centerline{\includegraphics{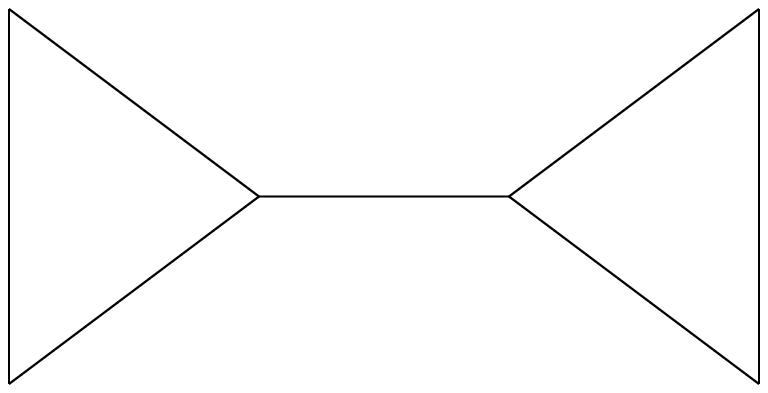}}
\centerline{Figure 2.1: Bowtie graph}
\vskip 5pt

The proof of \thmref{zcnt} provides a combinatorial basis construction for $E_0(4n^2\pi ^2)$.
A similar construction using even cycles will provide independent elements of $E_0((2n-1)^2\pi ^2)$,
but these may not form a complete set.  Consider the bowtie graph in Figure 2.1, 
whose cycles have length $3$.  An eigenfunction $\psi $ with eigenvalue $\pi ^2$
can be constructed by letting $\psi (x) = 2\sin (\pi x)$ on the middle edge.
The function $\psi $ then continues as $-\sin (\pi x)$ on the adjacent four edges,
and $\sin (\pi x)$ on the remaining two edges.  Notice that on each edge
the function $\psi $ is an integer multiple of $\sin (\pi x)$.   

A generalization of the even cycles will be used for a combinatorial construction 
of $E_0((2n-1)^2\pi ^2)$. 
Let $Z_1$ denote the set of functions $f:\edgeset \to \complex $
such that 
\[ \sum_{e \simeq v} f(e) = 0, \quad v \in \vertexset .\]

\begin{thm} \label{zcnt2}
The linear map taking $f \in Z_1$ to $g(x)$ defined by
\[ g_e(x) = f(e) \sin ((2n-1)\pi x),\] 
is an isomorphism
from  $Z_1$ onto $E_0((2n-1)^2\pi ^2)$.
The subspace $Z_1$ has an integral basis.
\end{thm}

\begin{proof}
Since $\sin ((2n-1)\pi x) = \sin ((2n-1)\pi (1-x)) $, 
the edge orientation does not affect the definition, so $g$ is well defined.   
Clearly $g(x)$ satisfies the eigenvalue equation and vanishes at each vertex.
The condition $ \sum_{e \simeq v} f(e) = 0$ for all $ v \in \vertexset $ gives
the derivative condition, so $g \in E_0((2n-1)^2\pi ^2)$.
Moreover the map is one to one.

Suppose $\psi (x) \in E_0((2n-1)^2\pi ^2)$.  Then 
$\psi _e(x) = a_e\sin((2n-1)\pi x)$ on each edge, and because $\psi $ satisfies the derivative
conditions we have 
\[ \sum_{e \simeq v} a_e = 0, \quad v \in \vertexset .\]
Define $f(e) = a_e$ to get a linear map from $E_0((2n-1)^2\pi ^2)$ to $Z_1$,
which is also one to one.  This establishes the isomorphism.

To see that $Z_1$ has an integral basis, let $e_n$ be a numbering of the edges of $\graph $,
and let $x_n = f(e_n)$.  The set of functions $Z_1$ is then given by the
set of $x_1,\dots ,x_{N_{\edgeset}}$ satisfying the $N_{\vertexset}$ equations 
\[ \sum_{e_n \simeq v} x_n  = 0.\]
This is a system of linear homogeneous equations whose coefficient matrix consists of 
ones and zeros.  Reduction by Gaussian elimination shows that the set of solutions 
has a rational basis, and so an integral basis.

\end{proof}

\section{Graph refinements}

Now we introduce the notion of graph refinement for graphs whose edge lengths are $1$.
Let the original combinatorial graph be denoted $\graph _1$, with the operator
$\Delta _1$ acting on the vertex space $\vertspace _1$.  
For each integer $N > 1$ we will define a graph $\graph _N$ with vertex space $\vertspace _N$
and operator $\Delta _N:\vertspace _N \to \vertspace _N$ by 
subdividing each edge $e \in \graph _1$ into $N$ edges.  Pick local coordinates identifying $e$ 
with $[0,1]$, and labelling $0 = x_0$ and $1 = x_N$.
For $n = 1,\dots ,N-1$, introduce points $x_n = n/N$ in the local coordinates. 
The new graph $\graph _N$ 
with $I = N_{\vertexset} + (N-1)N_{\edgeset}$ vertices
will have the vertex set consisting of the vertices of $\graph _1$,
together with the new vertices $x_n$, for each edge $e$ of $\graph _1$.  
The vertex $x_n$ is adjacent to $x_{n-1}$ and $x_{n+1}$ for $n = 1,\dots ,N-1$. 
If $v$ is a vertex in $\graph _1$, and the local coordinates for edges incident on $v$
are chosen so $v$ is identified with $0$ on each edge, then $v$ is adjacent in $\graph _N$
with the vertices $x_1$ for each of the incident edges.  An example is illustrated in Figure 3.1.

\vskip 8pt

\centerline{\includegraphics[scale=.75]{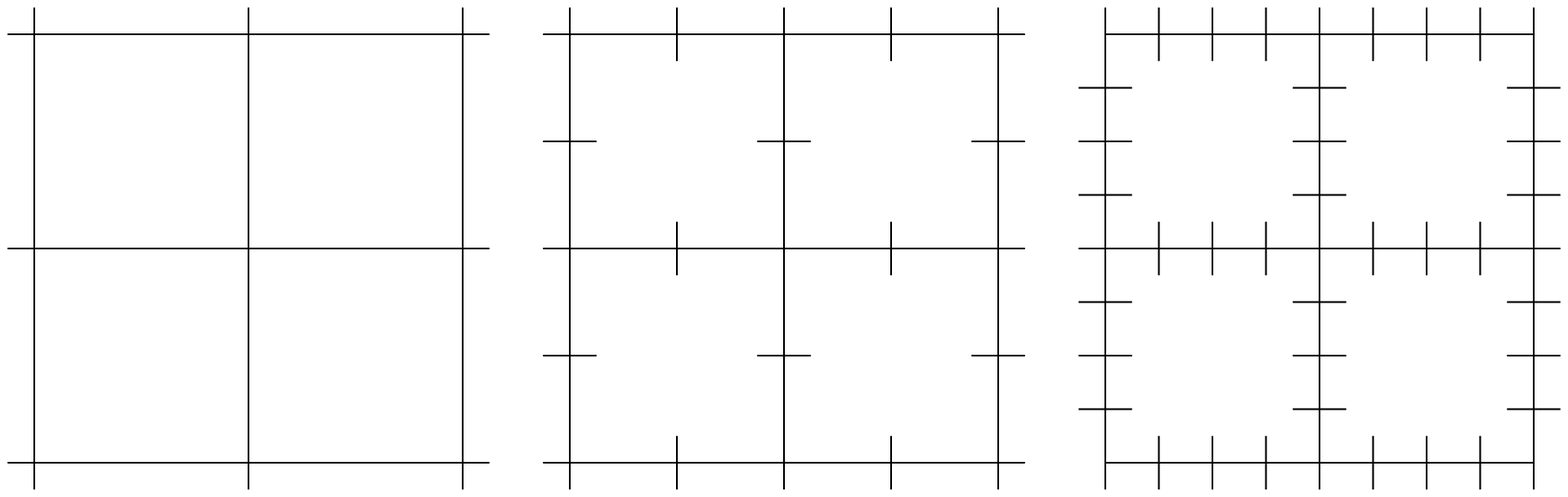}}
\centerline{ \hskip 5pt $\graph _1$ \hskip 122pt $\graph _2$ \hskip 122pt $\graph _4$ } 
\centerline{Figure 3.1: Refinement of a graph} 

\vskip 5pt

An inner product on the vertex space $\vertspace _N$ is defined by
\[ \langle f,g \rangle _{N} =
 \frac{1}{W}\sum_v deg(v) f(v)\overline{g(v)},
\quad W = \sum_{v \in \graph _N} deg(v) = 2NN_{\edgeset} .\]
Using the above identifications of edges of $\graph _N$ with subintervals of $[0,1]$,
we may also construct the corresponding continuous graph having edges of length $1/N$.
These continuous graphs may all be identified with the continuous graph $\graph _\infty$ 
corresponding to $\graph _1$.  In particular they will share the inner product 
\[ \langle f,g \rangle _{\infty } = \frac{1}{N_{\edgeset}}\int_{\graph _{\infty}} f\overline{g},\]
where $N_{\edgeset}$ is the number of edges in $\graph _1$.
As a continuous graph, the standard vertex conditions hold at all vertices of $\graph _N$.  
Although appearing different in definition, the continuous graph Laplacian $\Delta _{\infty}$
has not changed as we pass from $\graph _1$ to $\graph _N$.

Let $\Delta _N$ denote the following operator on the vertex space of $\graph _N$, 
\[ \Delta _N = N^2(I - T^{-1}A ).\]
The eigenspaces of these operators with eigenvalue $\lambda $ will be denoted
$E_N(\lambda )$, while $E_{\infty}(\lambda )$ will denote an eigenspace
for $\Delta _{\infty}$. 
By \eqref{rescale} and \thmref{combspec}, the mapping 
$N^2(1 - \cos(\sqrt{\lambda}/N))$ 
carries eigenvalues of $\Delta _{\infty}$ to eigenvalues of $\Delta _N$
if $\lambda /N^2 \notin \{ n^2 \pi ^2 \}$.
That is, the eigenvalues of the normalized adjacency
operator are $\cos(\sqrt{\lambda }/N)$ whenever $\lambda \notin \{ (Nn\pi )^2 \}$
is an eigenvalue for $\Delta _{\infty}$.
Notice the eigenvalues of $\Delta _N$ are
approximately $ \lambda /2 $ when $\lambda $ is small compared to $N$.

\begin{lem}
For $1 \le N < \infty$ the operator $\Delta _N$ is self adjoint on the vertex space
with the inner product $\langle f,g \rangle _{N} $.
\end{lem}

\begin{proof}
It suffices to check the normalized adjacency operator $T^{-1}A$.
\end{proof}

Let $E_p(n^2\pi ^2)$ denote the subspace spanned by eigenfunctions 
of $\Delta _{\infty}$ having the form
$\cos(n\pi x)$ on each edge (so not vanishing at the vertices). 
Let $\espace _N \subset L^2(\graph _{\infty})$ denote the subspace 
\[\espace _N = {\rm span} \{ E_p(N^2\pi ^2), E_{\infty}(\lambda ), 
0 \le \lambda < N^2\pi ^2 \} .\]
Here is preliminary result describing the restriction of eigenfunctions
of $\Delta _{\infty}$ to the vertices of $\graph _N$.

\begin{prop} \label{bigprop1}
The restriction map $R_N:\espace _N \to \vertspace _N$ is an bijection.
For $0 \le \lambda < N^2 \pi ^2$ this map takes distinct orthogonal eigenspaces 
$E_{\infty}(\lambda )$ of $\Delta _{\infty}$ onto distinct orthogonal eigenspaces 
$E_N( N^2(1 - \cos (\sqrt{\lambda }/N)))$ of $\Delta _{N}$,
and $R_N$ takes $E_p(N^2\pi ^2)$ onto $E_N(2N^2)$. 
\end{prop}

\begin{proof}
By \thmref{combspec} and \eqref{rescale}, $R_N$ is a bijection from  
the eigenspace $E_{\infty}(\lambda )$ of $\Delta _{\infty}$ to the eigenspace
$E_N( N^2(1 - \cos (\sqrt{\lambda }/N)))$ of $\Delta _{N}$.
The $0$ eigenspaces for $\Delta _{\infty}$ and $\Delta _N$ are just the
functions which are constant on the connected components of the respective graphs, so
$R_N$ is a bijection from $E_{\infty}(0)$ to $E_N(0)$.
From the proof of \lemref{nzcnt} we also see that 
$R_N$ is a bijection from $E_p(N^2\pi ^2)$ to $E_N(2N^2)$. 

Since $\Delta _{\infty}$ and $\Delta _{N}$ are self adjoint on their respective function spaces, 
and since $\cos (t)$ is strictly decreasing on $(0,\pi )$, distinct eigenspaces
of $\Delta _{\infty }$, which are orthogonal, are mapped to 
distinct eigenspaces of $\Delta _N$, which are orthogonal and span $H_N$.
\end{proof}

The restriction map $R_N:\espace _N \to \vertspace _N$ also has 
noteworthy features on the individual eigenspaces 
$E_{\infty}(\lambda )$ of $\Delta _{\infty}$.
Before stating the results, 
we start with a simple observation relating the $\vertspace _N$ 
inner product and the sums appearing in the trapezoidal rule for integration.
Recalling that $W = \sum_v deg(v) = 2NN_{\edgeset}$, the inner product for $\graph _N$ satisfies the identity
\[ W \langle f,g \rangle _{N} = \sum_v deg(v) f(v)\overline{g(v)} \]
\[ = \sum_{e \in \graph _1} [f_e(x_0)\overline{g_e(x_0)} + f_e(x_N)\overline{g_e(x_N)} +
2\sum_{n=1}^{N-1} f_e(x_n)\overline{g_e(x_n)}] .\]
These last sums are just the trapezoidal rule sums used for integrals over the edges of $\graph _{\infty}$.
With this motivation, if the continuous linear functional $T_N:C[0,1] \to \complex $ is defined by
\[T_N(f) = \frac{1}{2N}[f(x_0) + f(x_N) + 2\sum_{n=1}^{N-1} f(x_n)] ,\]
then
\begin{equation} \label{Nip}
\langle R_Nf,R_Ng \rangle _{N} = \frac{1}{N_{\edgeset}}\sum_{e \in \edgeset} T_N(f_e\overline{g_e}). 
\end{equation}

The following identities are also useful.  First
\[\int_0^1 e^{2i\omega x} = \frac{\exp(2i\omega ) - 1 }{2i\omega } = e^{i\omega }\frac{\sin(\omega )}{\omega }.\]
Then, for $0 < \omega < N$, a geometric series computation gives
\[2N T_N(\exp(i\omega x)) =  \exp(i\omega ) - 1 + 2\sum_{n=0}^{N-1}\exp(i\omega n/N) \]
\[ =  \exp(i\omega ) - 1 + 2\frac{1 - \exp(i\omega )}{1 - \exp(i\omega /N)} 
 = (1 - \exp(i\omega )) \frac{1 + \exp(i\omega /N)}{1 - \exp(i\omega /N)} ,\]
or
\begin{equation} \label{freqresp}
T_N(e^{i\omega x}) =  M_0(\frac{\omega }{2N} )\int_0^1 e^{i\omega x} , \quad 
M_0(z)  = M_0(-z) = z \cot (z). 
\end{equation}

These identities will be useful for comparing the inner products on $L^2(\graph _{\infty})$
and $\vertspace _N$.  The cases $\lambda = k^2\pi ^2$ with $0 \le k < N$ are
considered first.

\begin{thm} \label{bigthm1}
Suppose $f,g \in E_{\infty}(\lambda )$, with $0 \le \lambda \le N^2\pi ^2$. 
If $\lambda = k^2\pi ^2$ for an integer $k$ with $0 \le k < N$, then
\begin{equation} \label{ipform1}
\langle f,g \rangle _{\infty} = \langle R_Nf,R_Ng \rangle _{N}.
\end{equation} 
If $f,g \in E_p(N^2\pi ^2)$, then
\begin{equation} \label{ipform2}
\langle f,g \rangle _{\infty} = \frac{1}{2} \langle R_Nf,R_Ng \rangle _{N}.
\end{equation} 

\end{thm}

\begin{proof}

Suppose $f,g \in E_{\infty}(\lambda )$, with $0 \le \lambda < N^2\pi ^2$. 
On each edge $e$ we have
\begin{equation} \label{localform}
f_e = \alpha _ee^{i\omega x} + \beta _ee^{-i\omega x}, \quad 
g_e = \gamma _ee^{i\omega x} + \delta _e e^{-i\omega x}, \quad \omega ^2 = \lambda ,
\end{equation}
\[f_e\overline{g_e} =  \alpha _e\overline{\gamma _e} + \beta _e\overline{\delta _e} +
\alpha _e\overline{\delta _e} e^{2i\omega x} + \beta _e\overline{\gamma _e} e^{-2i\omega x}. \]

If $\lambda = k^2\pi ^2$ for an integer $k$ with $0 \le k < N$, then
\[ T_N(e^{2\pi ik x}) = \int_0^1 e^{2\pi ik x} = 0 .\]
Using \eqref{Nip} and $T_N(1) = 1$, it follows that
\[ \langle f,g \rangle _{\infty} = \frac{1}{N_{\edgeset}} \int_{\graph _{\infty}} f\overline{g} =  
\frac{1}{N_{\edgeset}} \sum _e [ \alpha _e\overline{\gamma _e} + \beta _e\overline{\delta _e}] \] 
\[ = \frac{1}{N_{\edgeset}} \sum _e T_N( \alpha _e\overline{\gamma _e} + \beta _e\overline{\delta _e}) 
= \langle R_Nf,R_Ng \rangle _N, \]
which is \eqref{ipform1}.  Similar calculations handle the case $\lambda = 0$.

Suppose $f,g \in E_p(N^2\pi ^2)$, a one dimensional space. For 
$f_e = g_e = \cos( N \pi x)$, we have
\[\frac{1}{N_{\edgeset}} \int_{\graph _{\infty}} \cos ^2(N \pi x) = 1/2, \]
but
\[ \langle R_Nf,R_Ng \rangle _{N} = \frac{1}{N_{\edgeset}} \sum_{e \in \edgeset} T_N(f_e\overline{g_e}) 
= \frac{1}{N_{\edgeset}} \sum_{e \in \edgeset} 1 = 1,\]
giving \eqref{ipform2}.

\end{proof}

For general $\omega $ a variation of a classical trapezoidal rule estimate appears.

\begin{thm} \label{bigthm2}
Suppose $f,g \in E_{\infty}(\lambda )$, with $0 < \lambda < N^2\pi ^2$. 
Then for $\omega \ge \omega _0 > 0$, 
\begin{equation} \label{ipform3}
\Big | \langle R_Nf,R_Ng \rangle _{N} - \langle f, g \rangle \Big | 
 = O( |\omega |^4/N^4) \| f \| \ \| g \|.
\end{equation} 

\end{thm}

\begin{proof}

The argument starts with a variation of standard error estimates 
\cite[p. 285]{Atkinson}, [p. 358-369]\cite{Briggs} for the trapezoidal rule for integration.
If $\phi = \exp(i\omega x)$ then
\[\frac{\phi '(1) - \phi '(0)}{12N^2} = \frac{1}{12N^2} \int_0^1 \phi '' 
= \frac{-\omega ^2}{12N^2} \int_0^1 e^{i\omega x}. \]
Putting this together with \eqref{freqresp} yields
\[\int_0^1 e^{i\omega x} 
- T_N(e^{i\omega x}) = - \frac{\phi '(1) - \phi '(0)}{12N^2} + 
M_1(\frac{\omega }{2N}) \int_0^1 e^{i\omega x}, \]
with
\[ M_1(z) = 1 - z\cot (z) - \frac{z^2}{3}.\] 
A Taylor expansion gives 
\[ M_1(z) = O(z^4), \]
and the function $M_1(z) = M_1(-z)$ is analytic for $|z| < \pi$.

On each edge we have the representation \eqref{localform}.
Since $f$ and $g$ are in the domain of $\Delta _{\infty }$, 
the function $\psi = f\overline{g}$ satisfies the vertex conditions 
\eqref{continuity} and \eqref{derive}.  Thus  
\[\sum_e [ \psi _e'(1) - \psi _e'(0) ]  
= - \sum_{v \in \vertexset} \sum_{e \sim v } \partial _{\nu} \psi _e(v) = 0, \]
and
\[N_{\edgeset} [\langle f, g \rangle - \langle R_Nf,R_Ng \rangle _{N} ] 
= \sum_e \Bigl [ \int_0^1 \psi _e - T_N(\psi_e) \Bigr ] \]
\[= - \sum_e \Bigl [ \frac{\psi _e'(1) - \psi _e'(0)}{12N^2} \Bigr ]  
 + M_1(\frac{\omega}{N})\Bigl [\int_0^1 e^{2i\omega x} \sum_e \alpha _e\overline{\delta _e}  
+ \int_0^1 e^{-2i\omega x} \sum_e \beta _e\overline{\gamma _e} \Bigr ] \]
\[ = M_1(\frac{\omega}{N})\Bigl [\int_0^1 e^{2i\omega x} \sum_e \alpha _e\overline{\delta _e}  
+ \int_0^1 e^{-2i\omega x} \sum_e \beta _e\overline{\gamma _e} \Bigr ] \]

Using $2|\alpha _e||\beta _e| \le |\alpha _e|^2 + |\beta _e|^2$, we find 
\[\int_0^1 |f_e|^2 = |\alpha _e|^2 + |\beta _e|^2 + \alpha \overline{\beta} e^{i\omega }\frac{\sin(\omega )}{\omega }
+ \beta \overline{\alpha} e^{-i\omega }\frac{\sin(\omega )}{\omega } \]
\[ \ge (1 - \frac{\sin(\omega )}{\omega })(|\alpha _e|^2 + |\beta _e|^2) .\]
This gives the desired result
\[ N_{\edgeset} \Big | \langle R_Nf,R_Ng \rangle _{N} - \langle f, g \rangle \Big | \]
\[ \le  \Big | M_1(\frac{\omega }{N}) (1- \frac{\sin(\omega )}{\omega })^{-1} 
2 \sum_e (\int_0^1 |f_e|^2)^{1/2} \int_0^1 |g_e|^2)^{1/2} \Big | \]
\[ \le  \Big | M_1(\frac{\omega }{N}) (1- \frac{\sin(\omega )}{\omega })^{-1} 
2 (\int_{\graph _{\infty}} |f|^2)^{1/2} (\int_{\graph _{\infty}} |g|^2)^{1/2} \Big | .\]

\end{proof}

\section{The FFT}

In this section we turn to an algorithm
for the efficient computation of the Fourier transform and its inverse transform on the spaces
$\vertspace _N$.  

\vskip 5pt
\centerline{\includegraphics{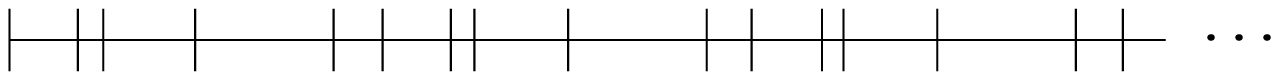}}
\[ 0 \hskip 47pt \pi \hskip 44pt 2\pi \hskip 41pt 3\pi \hskip 40pt 4\pi \hskip 42pt 5\pi \hskip 77pt \sqrt{\lambda } \]
\centerline{Figure 4.1: Square root of the spectrum of $\Delta _{\infty}$.}
\vskip 5pt

Recall some of the structures associated with the spectrum of $\Delta _{\infty}$.
The square roots $\sqrt{\lambda _k}$ of the positive eigenvalues of $\Delta _{\infty}$
and the eigenspaces $E(\lambda _k)$ exhibit the '$2\pi$ -periodicity' discussed in \propref{specper}. 
and illustrated in Figure 4.1.
Let $0 < \omega _{0,1} < \omega _{0,2} < \dots \le 2\pi $ denote the square roots of the $K$ distinct eigenvalues
$\lambda _k$ of $\Delta _{\infty}$ with $0 < \lambda _k \le (2\pi)^2$.  
For each $\omega _{0,k}$ pick an orthonormal basis $\Psi (\omega _{0,k})$ for $E(\omega _{0,k}^2)$, with
basis functions $\Psi _j(\omega _{0,k})$, $j = 1,\dots ,{\rm dim} E(\omega _{0,k}^2)$.

\[ \begin{matrix} 
  & \omega _{0,k} & & \omega _{1,k} & & \omega _{2,k} & & \omega _{3,k} & \cr
\cr
2 \pi  & \bullet & <---> & \bullet & <--->  & \bullet & <---> & \bullet & <--> \cr
  & \bullet & <---> & \bullet & <--->  & \bullet & <---> & \bullet & <--> \cr
\pi  & \bullet & <---> & \bullet & <--->  & \bullet & <---> & \bullet & <--> \cr
  & \bullet & <---> & \bullet & <--->  & \bullet & <---> & \bullet & <--> \cr
  & \bullet & <---> & \bullet & <--->  & \bullet & <---> & \bullet & <--> \cr
0  & \bullet  & &  & m  &  &  &  &  
\end{matrix}
\]
\centerline{Figure 4.2: Eigenfunction maps for $\Delta _{\infty}$.}
\vskip 5pt

For $m = 1,2,3,\dots $, let $\omega _{m,k} = \omega _{0,k} + 2m\pi $.
From each basis $\Psi (\omega _{0,k})$ we may produce bases $\Psi (\omega _{m,k})$
for $E(\omega _{m,k}^2)$ by the method of \propref{specper}. 
Recall that on each edge $e$ a function $\Psi _j \in \Psi (\omega _{0,k})$ 
has the form 
\[\Psi _j = A_e\exp(i\omega _{0,k} x) + B_e \exp(-i \omega _{0,k} x) , \quad 0 \le x\ \le 1. \]
The set of functions $\Psi _j(\omega _{m,k})$ whose values on $e$ are 
\begin{equation} \label{psiform}
\Psi _j(\omega _{m,k}) = A_e\exp(i[\omega _{0,k} +2m\pi ]x) + B_e\exp (-i[\omega _{0,k}  + 2m\pi ]x) ,  
\end{equation}
then give a basis $\Psi (\omega _{m,k})$ for $E([\omega _{0,k}  + 2m\pi ]^2)$.
This basis need not be orthonormal.  Figure 4.2 indicates the relations among the 
bases $\Psi (\omega _{m,k})$.

\begin{prop} \label{orthcases}
If $\omega _{0,k}  \in \{ \pi , 2\pi \}$, then the functions $\Psi _j(\omega _{m,k})$ 
are orthonormal, as are the functions $R_N\Psi _j(\omega _{m,k}) \in \vertspace _N$ for
$ 0 \le \omega _{m,k} < N \pi  $. 
\end{prop}

\begin{proof}
Suppose $\Psi _i(\omega _{m.k}) \in \Psi (\omega _{m,k})$ is represented on the edge $e$ by
\[\Psi _i(\omega _{m,k}) = C_e\exp(i[\omega _{0,k} +2m\pi ]x) + D_e\exp (-i[\omega _{0,k}  + 2m\pi ]x) , \quad 0 \le x\ \le 1, \]
and \eqref{psiform} describes $\Psi _j$.  With $\omega _{0,k} = \pi , 2\pi ,$
\[ \int_0^1 \exp(\pm 2i[\omega _{0,k}  +2m\pi ]x) = 0,\]
so
\[ \sum_e \int_0^1 \Psi _j(\omega_{m,k})\overline{\Psi _i(\omega _{m,k})} 
 =  \sum_e [A_e \overline C_e + B_e \overline D_e] 
=  \sum_e \int_0^1 \Psi _j(\omega _{0,k}) \overline{\Psi _i(\omega _{0,k})}.\]

The extension to $\vertspace _N$ follows from \thmref{bigthm1}.
\end{proof}

Let $I_0 = {\rm dim}E(0)$, $I_k = {\rm dim}E(\omega _{0,k}^2)$, 
and let
\[ \Phi _j(\omega _{m,k})  = R_N\Psi _j(\omega _{m,k}) \]
be the restriction of $\Psi _j(\omega _{m,k})$ to the vertices of $\graph _N$. 
The bases $\Phi (\omega _{k,m})$ will be used for 
efficient Fourier transform algorithms 
\cite[p. 182]{Atkinson} \cite[p. 383]{Briggs}.
Define the discrete Fourier transforms (DFT)
$ \ft _N: \vertspace _N \to \complex ^{I_0} \oplus _{m,k} \complex ^{I_k} $
for the continuous graph $\graph _{\infty}$
by taking the inner product of a function $f \in \vertspace _N$
with the bases $\Phi (\omega _{m,k})$, 
\begin{equation} \label{DFT}
\ft _N(f) = \{ \langle f, \Phi _j(\omega _{m,k}) \rangle _N \} ,\quad \omega _{m,k} \le N\pi .
\end{equation}
The condition $\omega _{m,k} \le N\pi $ amounts to $m = 0,\dots , N/2 -1$ when $N$ is even,
which will hold in the cases of interest below.
Except for the cases noted in \propref{orthcases}
the functions of $\Phi (\omega _{m,k})$ may not be orthonormal, 
so the Fourier transform will not be an isometry from 
$\Phi (\omega _{m,k})$ to $\complex ^{I_k}$ if the usual inner product is used on the range.
We consider a modified inner product.  

Suppose $B(\omega _{m,k}) = (b_{ij})$ is a matrix taking the basis $\Phi (\omega _{m,k})$
to an orthonormal basis 
\begin{equation} \label{cob}
\eta _i = \sum_j b_{ij}\Phi _j(\omega _{m,k}) 
\end{equation}
for the same eigenspace in $\vertspace _N$. 
(Such a matrix may be obtained by the Gram-Schmidt process.) 
For $f$ in the span of $\Phi (\omega _{m,k})$, the map
$f \to \{ \langle f, \eta _i \rangle _N \}$
is an isometry to $\complex ^{I_k}$ with the usual inner product
$X \bullet Y = \sum x_j\overline{y_j}$.  
Expressing this in the original basis, 
\[ \| f \| ^2 = \sum_i | \langle f, \eta _i \rangle _N|^2
= \sum_i |\langle f, \sum_j b_{i,j} \Phi _j  \rangle _N |^2 
 = \sum_i  | \sum_j \overline{b_{i,j}} \langle f, \Phi _j  \rangle _N |^2 \]
\[ = \overline{B}X\bullet \overline{B}X, \quad 
X = ( \langle f, \Phi _1 \rangle _N , \dots ,\langle f, \Phi _{I_k} \rangle _N )^T. \]

For $f = \sum_j c_j\Phi _j(\omega _{m,k})$ in the span of $\Phi (\omega _{m,k})$, the coefficients $c_j$ 
may be recovered from the DFT values $\langle f, \Phi _j \rangle _N $.
Starting from 
\[f = \sum_i a_i\eta _i , \quad a_i = \langle f, \eta _i \rangle _N ,\]
and using \eqref{cob} we find
\[ f = \sum_i (\sum_j \overline{b_{ij}}\langle f,\Phi _j \rangle _N ) (\sum_l b_{il}\Phi _l ) 
 = \sum_l \sum_j (\sum_i \overline{b_{ij}} b_{il} \langle f,\Phi _j  \rangle _N)  \Phi _l  \]
\[ = \sum_l ( \sum_j (B^*B)_{jl} \langle f,\Phi _j \rangle _N ) \Phi _l , \]
or
\begin{equation} \label{cval}
c_l =  \sum_j (B^*B)_{jl} \langle f,\Phi _j \rangle _N . 
\end{equation}

Calling $\overline{B}X\bullet \overline{B}X$ the $B(\omega _{m,k})$ inner product 
on $\complex ^{I_k}$, and noting $\Phi (\omega _{m,k})$ is a basis for the
$\mu _{m,k}$ eigenspace of $\Delta _N$, we obtain the first part of the next result. 

\begin{thm}
The Fourier transform $ \ft _N: \vertspace _N \to \complex ^{I_0} \oplus _{m,k} \complex ^{I_k} $
satisfies
\[\ft _N(\Delta _Nf) = \{ \mu _{m,k}\ft (f)_{m,k} \}, \]
and is an isometry if $\complex ^{I_k}_m$ has the $B(\omega _{m,k})$ inner product. 

If $N$ is a power of $2$, then
$\ft _N(f)$ can be computed in time $O(N\log_2(N))$.  
\end{thm}

\begin{proof}

The inner products of \eqref{DFT} may be split into trapezoidal sums for the edges as in \eqref{Nip}, 
\[ \ft _N(f) = \frac{1}{N_{\edgeset}} \sum_e T_N(f_e \overline{\Phi _j(\omega _{m,k})}) .\]
Using the representation \eqref{psiform} we have
\[T_N(f_e \overline{[A_e \exp (i[\omega _{0,k} + 2m\pi ]x) + B_e \exp (-i[\omega _{0,k} + 2m\pi ]x)]}) \]
\[= \overline{A_e} T_N(\exp(-i\omega _{0,k}x)f_e \exp (-2m\pi ix) 
+ \overline{B_e}T_N(\exp(i\omega _{0,k}x)f_e \exp (2m\pi ix)) . \]

If $N$ is a power of $2$,
$e$ and $k$ are fixed, and $m = 0,\dots ,\frac{N}{2}-1$, then  
the sequence $T_N(\exp(\pm i\omega _{0,k}x)f_e(x) \exp (\pm 2m\pi ix))$
differs trivially from the conventional discrete Fourier transform
of the sampled data $f_e(x_n)$,  
which may be computed \cite[p. 182]{Atkinson} \cite[p. 383]{Briggs} in time $O(N\log_2(N))$. 
Since the number of edges $e$ and indices $k$ is fixed, we obtain the result.

\end{proof}

It remains to consider the inverse transform.  
Since the matrix $B^*B$ of \eqref{cval} depends only on the graph and functions 
$\Phi _j(\omega _{m,k})$, these matrices are assumed to be precomputed.  
Suppose the DFT sequence \eqref{DFT} of a function $f$ is given.  
Since the computations in \eqref{cval} are limited to sequence blocks $\complex ^{I_k}$,
the coefficients $c_j(\omega _{m,k})$ for all $\omega _{m,k}$ may be computed in time $O(N)$.
It remains to compute
\[ f(x_n) = \sum_{m,k} c_j(\omega _{m,k})\Phi _j(\omega _{m,k})(x_n). \]

Fixing the edge $e$, the primitive frequency index $k$, and the basis index $j$, 
and taking advantage of \eqref{psiform},
the remaining sum has the form
\[ f_{e,k,j}(x_n) = A_e \exp(i\omega _{0,k}n/N) \sum_m c_j(\omega _{m,k}) \exp(i2m\pi n/N) \]
\[ + B_e \exp(-i\omega _{0,k}n/N) \sum_m c_j(\omega _{m,k}) \exp(-i2m\pi n/N) .\]
These sums may be computed with a conventional FFT.
Since the number of edges $e$
and primitive frequency indices $k$ are constant, 
and the range of basis indices $j$ is bounded independent of $m$,
the next result is obtained.

\begin{thm}
If $N$ is a power of $2$, then
the inverse DFT can be computed in time $O(N\log_2(N))$.  
\end{thm}

\section{A family of examples}

A relatively simple family of examples is provided by  
the complete bipartite graphs $K(m,2)$ on $m,2$ vertices.
The $m=4$ case is portrayed in Figure 5.1.  
Discussion of these graphs is facilitated if the vertices are colored,
the $2$ vertices of degree $m$ being red and the remaining $m$ vertices blue.

Suppose $v$ is a vertex, and $\{ u_i \} $ are the vertices adjacent to $v$.
Then the operator $\Delta _1$ 
on the $m+2$ dimensional vertex space is given by
\[ \Delta _1 f(v) = f(v) - \frac{1}{deg(v)}\sum_{i = 1}^{deg(v)} f(u_i). \]
Eigenvalues $\mu $ for $\Delta _1$ are $0,1,2$, with $1$ having multiplicity $m$.  
The eigenfunctions $\psi $ for $\mu = 0 ,2$ are 
\[ \mu _0 = 0, \quad \psi _0(v) = 1, \] 
\[ \mu _{m+1} = 2, \quad 
\psi _{m+1}(v) = \Bigl \{ \begin{matrix} 1, \quad v {\ \rm red} \cr 
-1, \quad v {\ \rm blue}  
\end{matrix} \Bigr \} \] 

\vskip 5pt
\centerline{\includegraphics{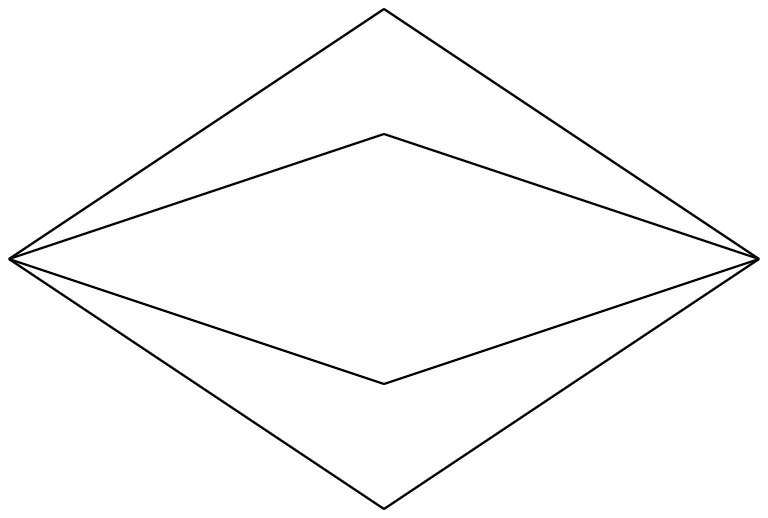}}
\centerline{Figure 5.1: The graph $K_{4,2}$}
\vskip 5pt

Next, consider the eigenfunctions for $\mu = 1$.
A function $f(v)$ will satisfy $\Delta  _1f(v) = f(v)$
if and only if $\sum_{i = 1}^{deg(v)} f(u_i) = 0$ for every vertex $v$.
Number the red vertices $r_1$ and $r_2$, and the two blue vertices $b_1,\dots ,b_m$.
For $j = 1,\dots ,m-1$, independent eigenfunctions are
\[ \psi _{j}(v) = \Bigl \{ \begin{matrix} 1, \quad v = b_j \cr 
-1, \quad v = b_{j+1}, \cr
 0, \quad {\rm otherwise}  
\end{matrix} \Bigr \} .\]
There is one additional independent eigenfunction 
\[ \psi _{m}(v) = \Bigl \{ \begin{matrix} 1, \quad v = r_1, \cr 
-1, \quad v = r_2, \cr
 0, \quad {\rm otherwise}  
\end{matrix} \Bigr \} .\]
To establish the independence of $\psi _1,\dots ,\psi _m$, note that if
\[\sum_{j=1}^m c_j\psi _j(v) = 0,\]
then $c_m = 0$ by evaluation at $r_1$, and evaluation at the points $b_j$ leads
to equations
\[c_1 = 0,c_1 = c_2,\dots ,c_{m-1} = c_m.\]

Moving to the continuous graph, identify each edge with $[0,1]$ 
so the blue vertex is identified with $0$. 
Consider the eigenvalues $\lambda $ and corresponding eigenfunctions 
$\Psi $ for $\Delta _\infty $,  
focusing on $\sqrt{\lambda } \in [0,2\pi ]$.  
Of course we have $\lambda _0 = 0$ with the constant eigenfunction $\Psi (0) = 1$.  
From $\mu _j = 1$ we obtain eigenvalues
\[ \sqrt{\lambda } = \cos^{-1}(1 - \mu _j) = \cos^{-1}(0) =
\pi /2, 3\pi /2 .\]
That is, $\lambda = (\pi /2)^2$ and $\lambda = (3\pi /2)^2$ are eigenvalues,
each with multiplicity $m$.  The corresponding eigenfunctions 
$\Psi _j(\pi /2)$ and $\Psi _j(3\pi /2)$  
are obtained by extrapolation as in \lemref{endmap} from the vertex values of $\psi _j$.

This graph \cite[p. 53]{Bollobas} has $m-1$ independent cycles, all of which are even.
One basis of cycles $C_i$ is given by the vertex sequence $r_1,b_i,r_2,b_{i+1},r_1$.
Each $C_i$ supports eigenfunctions $\Psi _i(\pi)$ and $\Psi _i(2\pi )$ 
which are respectively $\pm \sin (\pi x)$ and $\pm \sin (2\pi x)$ on the edges of the cycle.

Finally, there are two additional independent eigenfunctions, 
$\Psi _m(\pi )$ having values $\cos(\pi x)$ 
and $\Psi _m(2\pi)$ with values $\cos(2\pi x)$ on the edges of $\graph $.
This gives a total of 
\[2N_{\edgeset} = 4m = 2m + 2(m-1) + 2 \]
independent eigenfunctions for $\Delta _{\infty}$ with eigenvalues $0 < \lambda \le 2\pi $.

This graph has an obvious involution in which $r_1$ and $r_2$ are interchanged, 
and we may split the eigenspaces into 
even and odd subspaces with respect to the involution.
Notice that the eigenfunctions $\Psi _m(\pi )$, $\Psi _m(2\pi )$, 
$\Psi _i(\pi /2)$ and $\Psi _i(3\pi /2)$ for $i = 1,\dots ,m-1$ are even 
while $\Psi _m(\pi /2)$, $\Psi _m(3\pi /2)$, $\Psi _i(\pi )$   
and $\Psi _i(2\pi )$ for $i = 1,\dots ,m-1$ are odd.
By restricting to the even, respectively odd, subspaces we may treat the case of a star graph
with $m$ edges and boundary vertices with the conditions $f'_e(0) = 0$, respectively $f_e(0) = 0$.

\vfill \eject

\end{document}